\documentclass[preprint]{elsarticle}

\usepackage{amsmath}
\usepackage{amssymb}
\usepackage{amsfonts}
\usepackage{amsthm}

\usepackage{graphicx}

\usepackage{algorithm}
\usepackage{algpseudocode}

\usepackage{hyperref}

\journal{Journal of Computational and Applied Mathematics}

\def\1{{\bf 1}}

\def\L{{\cal L}}

\def\LL{{\sf L}}
\def\DD{{\sf D}}
\def\AA{{\sf A}}

\def\MM{{\sf M}}

\def\II{{\sf I}}
\def\b{{\sf b}}
\def\k{{\sf k}}

\def\x{{\sf x}}
\def\y{{\sf y}}
\def\e{{\sf e}}
\def\w{{\sf w}}
\def\z{{\sf z}}

\newcommand{\Sparsify}{{\sf Sparsify}\kern1pt}
\newcommand{\SparsifiedKemeny}{{\sf SparsifiedKemeny}\kern1pt}
\newcommand{\SparsifiedMultigrid}{{\sf SparsifiedMultigrid}\kern1pt}

\newtheorem{proposition}{Proposition}[section]

\newtheorem{theorem}[proposition]{Theorem}
\newtheorem{corollary}[proposition]{Corollary}
\newtheorem{example}[proposition]{Example}

\begin{document}

\begin{frontmatter}

\title{Approximating Kemeny's constant\\ via spectral sparsification}

\author[tue,vub]{A. Abiad}
\ead{a.abiad.monge@tue.nl}

\author[upc]{\'A. Carmona}
\ead{angeles.carmona@upc.edu}

\author[upc]{A. M. Encinas}
\ead{andres.marcos.encinas@upc.edu}

\author[upc]{M. J. Jim\'enez}
\ead{maria.jose.jimenez@upc.edu}

\author[cunef]{\'A. Samperio\corref{cor1}}
\ead{alvaro.samperio@cunef.edu}

\cortext[cor1]{Corresponding author}

\address[tue]{Department of Mathematics and Computer Science,
Eindhoven University of Technology, Eindhoven, The Netherlands}

\address[vub]{Department of Mathematics and Data Science,
Vrije Universiteit Brussel, Brussels, Belgium}

\address[upc]{Departament de Matem\`atiques,
Universitat Polit\`ecnica de Catalunya, Barcelona, Spain}

\address[cunef]{Departamento de Matem\'aticas,
CUNEF Universidad, Madrid, Spain}

\begin{abstract}
Kemeny’s constant of a network quantifies the expected time for a random walk to reach a target vertex chosen according to the stationary distribution, independently of the starting vertex. The purpose of this work is twofold. First, we introduce the class of slightly regular networks, whose orthogonal degree vector is an eigenvector of the Laplacian, and derive an exact expression for Kemeny’s constant for this class, extending previous results beyond regular networks. Second, we address the efficient approximation of Kemeny’s constant using spectral sparsification. We establish multiplicative bounds relating Kemeny's constant of a network to that of any of its spectral sparsifiers, and use these bounds to design an approximation algorithm. While standard exact approaches based on a full spectral decomposition or on the computation of a pseudoinverse have \(O(n^3)\) complexity, our method has complexity \(O(n^2\log^{1+o(1)}n)\) for fixed approximation parameters, and is therefore nearly quadratic in the number of vertices even for dense networks.
\end{abstract}


\end{frontmatter}

\section{Introduction}

Kemeny’s constant is a fundamental parameter in the theory of Markov chains that has received renewed attention in graph theory and network science. Originally defined for finite, discrete-time, time-homogeneous, irreducible Markov chains in terms of the stationary distribution and mean first passage times, it is particularly relevant in the study of random walks on networks. It quantifies the expected time needed for a random walker to reach a target vertex selected according to the stationary distribution, independently of the starting vertex.

Kemeny’s constant has found applications in several areas, including \\epidemic-network analysis and testing strategies \cite{Detal2023,Yetal2020}, the identification and optimization of metastable states in molecular kinetic networks \cite{MMLR2024}, and the analysis of road-network dynamics \cite{CKS2011}. More generally, it is often interpreted as a global measure of network navigability, with smaller values typically associated with more efficient exploration of the network; see, for instance, \cite{BCK2022,BK2019,kim2023bounds,KLMZ2024}.

Several expressions for Kemeny’s constant are known in the literature; see, e.g., \cite{CaJiMa23,Lo16,Wa17}. Our first contribution is an exact expression for Kemeny’s constant of slightly regular networks (Proposition \ref{propo:slightly}), namely, networks whose orthogonal degree vector is an eigenvector of the Laplacian. This expression extends \cite[Proposition 1]{KoDu20} to a broader class of networks. We also identify classes of semiregular networks satisfying the required property, including complete bipartite graphs and several generalized windmill constructions.

Standard exact approaches for computing Kemeny’s constant, based on a full spectral decomposition or on the computation of a matrix pseudoinverse, exhibit an \(\mathcal{O}(n^3)\) complexity for a network with \(n\) vertices. This becomes computationally expensive for large networks and has motivated the development of approximation methods with lower computational cost \cite{BDKM24, LHL2021}.

Kooij and Dubbeldam \cite{KoDu20} proposed two approximations for non-regular graphs based on the Kirchhoff index and degree heterogeneity. However, a direct computation of the Kirchhoff index through standard dense spectral or pseudoinverse methods also has cubic complexity. Bini {\it et al.} \cite{BDKM24} derived decomposition formulas based on stochastic complements and used them to develop a divide-and-conquer strategy for computing Kemeny’s constant. They also investigated randomized trace estimation using Hutch++ combined with iterative linear-system solvers. Although they do not derive an overall asymptotic complexity bound for the divide-and-conquer algorithm comparable to the one considered here, their numerical experiments show computational savings over direct computation for most of the tested instances, although the improvement is not uniform across all networks. Li, Huang and Lee \cite{LHL2021} proposed a Monte Carlo estimator with probabilistic accuracy guarantees and developed a parallel GPU implementation for large-scale networks.

In \cite{XZ24}, Xia and Zhang introduced two Monte Carlo algorithms which compute approximations to Kemeny's constant, called {\it RefinedMC} and {\it ForestMC}. The computational complexity of the RefinedMC algorithm is asymptotically sublinear in the number of vertices, but it is asymptotically quartic in a parameter $l$ that goes to infinity if the absolute value of the second or last eigenvalues of the normalized adjacency matrix (which we denote as $|\lambda_2|$ and $|\lambda_n|$ in Section \ref{sec:preliminaries}) approach $1$. Similarly, the computational cost of the ForestMC algorithm is asymptotically cubic in a parameter $l$ that also goes to infinity if $|\lambda_2|$ or $|\lambda_n|$ approach $1$. Therefore, the complexity of those algorithms depends heavily on the structure of the network, and while they could have a very small computational cost for some families of networks, they could be prohibitively computationally costly for others. Xu {\it et al.}  \cite{XSZKZ20} proposed ApproxKemeny, a theoretically guaranteed approximation algorithm for undirected graphs based on Laplacian linear systems and Hutchinson trace estimation. They approximate this expression by solving only $\mathcal{O}\big(\log n\big)$ linear systems in the SDD matrix. Using the approximate SDD solver introduced in \cite{KyngSachdeva2016}, they obtain their approximation to the Kemeny constant with computational cost $\mathcal{O}\big(m\log^5(n)\big)$, where $m$ is the number of edges. While this approach has a small computational cost for sparse graphs, for dense graphs the computational cost becomes 
$\mathcal{O}\big(n^2\log^5(n)\big)$.

 The second contribution of our paper is the development of Algorithm \ref{sparsifiedkemeny}, which calculates a novel approximation for the Kemeny constant of any network. For dense networks, its worst-case asymptotic complexity improves on that of the comparison algorithms considered below. We should also note that the worst-case asymptotic complexity of our algorithm does not depend on the structure of the network. Our main tools are known techniques of spectral graph sparsification (see \cite{KLP12, SpSr2011}), which aim to find a sparse subgraph whose Laplacian matrix can well approximate the original Laplacian matrix in terms of its quadratic form and, therefore, also its eigenvalues. We prove that the Kemeny constant of the sparse approximation is close to the Kemeny constant of the original one (see Corollary \ref{propo:direct}).  We then combine this sparsification step with the representation of Kemeny’s constant in terms of Laplacian linear systems used in \cite{XSZKZ20}. For the complexity analysis, these systems are solved using the Sparsified Multigrid solver introduced in \cite{KLPSS16}.

Using this strategy of solving all linear systems after sparsifying the network, we achieve a computational cost of $\mathcal{O}\big(n^2 \log^{1+o(1)}(n)\big)$, even for dense graphs, which is nearly quadratic in the number of vertices (see Theorem \ref{theor:coste} for a proof of the expression of the computational cost as a function of $n$ and the two parameters that control the error of our approximation, $\beta$ and $\epsilon$). 
Also, as we avoid using the approximation based on Hutchinson’s Monte-Carlo method used in \cite{XSZKZ20}, we obtain
different accuracy guarantees. We identify a broad range of the approximation parameters for which our bounds are tighter than those of \cite{XSZKZ20} (see Theorem \ref{theor:aprox} and the discussion on the comparison between the bounds of both approximations right before Subsection \ref{sec:comp}).

To conclude, our approach is designed for the dense-network regime, where sparsification can substantially reduce the number of edges before the Laplacian systems are solved. We show that, for fixed approximation parameters, the resulting complexity is nearly quadratic in the number of vertices, and we compare both its complexity and its approximation guarantees with representative existing methods. 

The two parts of the paper are connected by a common spectral viewpoint. In Section 3, we use Rayleigh-quotient arguments to identify when the degree-heterogeneity term appearing in the Laplacian representation of Kemeny’s constant is spectrally exact. This leads naturally to the class of slightly regular networks and explains several equality cases for previously known bounds. In Section 4, we use the same variational perspective in an approximation setting: spectral sparsification provides multiplicative control of Laplacian quadratic forms, which in turn yields control of the normalized Laplacian spectrum and of Kemeny’s constant.

\section{Preliminaries}\label{sec:preliminaries}

In this paper, we consider connected networks $G=(V,E,c)$, where
$(V,E)$ is a simple undirected graph with $|V|=n$ vertices and $|E|=m$ edges, and
\(
c:V\times V\longrightarrow [0,\infty)
\)
is a symmetric function satisfying $c(i,j)>0$ if and only if
$\{i,j\}\in E$. We write $c_{ij}=c(i,j)$ and, for an edge
$e=\{i,j\}\in E$, we also use the notation $c(e)=c_{ij}$. Recall that the \emph{degree} of vertex $i$ is $k_i=\sum\limits_{j=1}^nc_{ij}$, and we define \emph{the volume} of $G$ by ${\rm vol}(G)=\sum\limits_{i=1}^nk_i.$ We denote 
\[
k_{\min}:=\min_{1\le i\le n} k_i,
\qquad
k_{\max}:=\max_{1\le i\le n} k_i.
\] When $c_{ij}=1$ for any $\{i,j\}\in E$, we refer to $G$ as a graph; in this case,   ${\rm vol}(G)=2m.$

The \emph{adjacency matrix} of the network $G=(V,E,c)$ is the $n\times n$ symmetric matrix $A=(c_{ij})$.  We denote by ${\sf e}
_i$ the $i$-th vector of the standard basis. 

Throughout this paper, vectors are represented using {\it sans serif} style. For instance, $\k$ is the vector of degrees of vertices of $G$. A network is called {\it regular} if the degree vector is constant. Additionally, for any real number $\alpha$, we define $\alpha^{\#}$ as $\dfrac{1}{\alpha}$ when $\alpha\not=0,$ and $0$ otherwise.  For a real symmetric matrix \(M\), we denote its group inverse by \(M^\#\). In this case, \(M^\#\) coincides with the Moore–Penrose inverse of \(M\).

The {\it Laplacian matrix} $\LL_G$ of $G$ is the $n\times n$ symmetric matrix given by $\LL_G=\DD-\AA$, where $\DD=\text{diag}(k_1,\ldots,k_n)$ is the diagonal degree matrix. The \emph{normalized Laplacian matrix} of a network, denoted ${\cal{L}}_G$, is defined as
$${\cal{L}}_G=\DD^{-1/2}\LL_G\DD^{-1/2}=\II-\DD^{-1/2}\AA\DD^{-1/2},$$
where $\DD^{-\frac{1}{2}}\AA\DD^{-\frac{1}{2}}$ is known as the {\it normalized adjacency matrix}. 

The random walk on \(G\) has transition probability matrix \(P=D^{-1}A\). Since \(G\) is connected, \(P\) is irreducible, and its unique stationary distribution \(\boldsymbol{\pi}\) is given by
$\pi_i=\dfrac{k_i}{{\rm vol}(G)},$  $i=1,\ldots,n$.
 
Kemeny's constant is a fundamental parameter in the study of Markov chains and random walks on graphs. It quantifies the expected number of steps a random walker takes to reach a randomly chosen vertex, regardless of the starting position, under the stationary distribution. Intuitively, Kemeny's constant provides a global measure of how efficiently information or influence propagates through a network. It is particularly useful in assessing the connectivity and mixing properties of a graph. Kemeny and Snell \cite{KeSn60} established a direct connection between Kemeny's constant and random walks. For $i\neq j$, let $m_{ji}$ denote the mean first passage time from
$j$ to $i$. For $i=j$, we adopt the recurrence-time convention
\[
m_{jj}=\frac{1}{\pi_j}.
\]
With this convention,
\[
\widetilde K(G)=\sum_{i=1}^n \pi_i m_{ji}
\]
is independent of the starting vertex $j$. Throughout this paper, we use
the shifted convention
\[
K(G):=\widetilde K(G)-1,
\]
which is the convention corresponding to the spectral expression below.

It is known that the Kemeny constant can be expressed by means of the eigenvalues of several matrices associated with the given network. In particular, there is a way to compute Kemeny's constant by using the eigenvalues of the
probability transition matrix. 

Since $P=D^{-1}A$ is similar to the symmetric normalized adjacency matrix
\(
D^{-1/2}AD^{-1/2},
\)
its eigenvalues are real. We denote them by
\[
1=\lambda_1>\lambda_2\ge \cdots \ge \lambda_n\ge -1.
\]
Accordingly, the eigenvalues of the normalized Laplacian
\(
\mathcal{L}_G
\)
are denoted by
\[
0=\mu_1<\mu_2\le \cdots \le \mu_n\le 2,
\]
and satisfy
\(
\mu_i=1-\lambda_i,\qquad i=1,\ldots,n.
\)
Finally, we denote the eigenvalues of the combinatorial Laplacian $\LL_G$
by
\[
0=\gamma_1<\gamma_2\le \cdots \le \gamma_n.
\]
Kemeny's constant admits the spectral representation (\cite{LL02})
\begin{equation}\label{Kemeny:eign}
K(G)
=
\sum_{i=2}^{n}\frac{1}{1-\lambda_i}
=
\sum_{i=2}^{n}\frac{1}{\mu_i}
=
\operatorname{tr}\!\left(\mathcal{L}_G^{\#}\right).
\end{equation}
 Furthermore, Wang et al. \cite[Theorem 4]{Wa17} showed that Kemeny's constant can also be expressed in terms of the group inverse of the combinatorial Laplacian, specifically, 

\begin{equation}  \label{thm:wangetal}
K(G)={\rm tr}(\LL_G^\# \DD)-\dfrac{1}{\rm{vol}(G)}{\k}^{\sf T}\LL_G^{\#}{{\sf k}}.
\end{equation}

An alternative derivation of this identity was given in \cite{CaJiMa23}, and related formulas were subsequently extended to Schr\"odinger random walks in \cite{CaEnJiMa24}.
\section{Rayleigh-quotient bounds and slightly regular networks}
\label{sec:slightly}

Several bounds for Kemeny's constant can be obtained from its
representation in terms of the group inverse of the combinatorial
Laplacian. In particular, the bounds derived in~\cite{Wa17} depend on
the interaction between the Laplacian spectrum and the deviation of
the degree vector from regularity. This raises a natural question:
under which structural or spectral conditions do these bounds become
exact, and can such conditions lead to an explicit expression for
Kemeny's constant beyond the regular case?

In this section, we show that the answer is governed by the spectral
position of the component of the degree vector orthogonal to the
all-ones vector. The equality cases of the known bounds correspond to
this vector lying in an extremal Laplacian eigenspace. This observation
naturally suggests considering the more general situation in which it
belongs to an arbitrary Laplacian eigenspace. We introduce the class of
slightly regular networks to capture this property and derive an exact
formula for Kemeny's constant on this class. We then identify broad
families of semiregular networks satisfying the condition, thereby
providing a common spectral explanation for several previously known
equality cases.

The starting point of our analysis is the representation of Kemeny's
constant in~\eqref{thm:wangetal}. For a network $G=(V,E,c)$, we define
the \emph{orthogonal degree vector} by
\[
\w=\k-\frac{\operatorname{vol}(G)}{n}\mathbf{1}.
\]
Equivalently, its $i$-th component is
\(
w_i=k_i-\frac{\operatorname{vol}(G)}{n},  i=1,\ldots,n.
\)
Since
\(
\mathbf{1}^{\top}\k=\operatorname{vol}(G),
\)
we have $\w\in\mathbf{1}^\perp$; hence $\w$ is precisely the component of
the degree vector orthogonal to the all-ones vector. Moreover,
$\w=\mathbf{0}$ if and only if $G$ is regular. Since $\LL_G^\#\mathbf{1}=\mathbf{0}$,
we have
\[
\k^\top\LL_G^\#\k
=
\w^\top\LL_G^\#\w.
\]
Hence, \eqref{thm:wangetal} can be rewritten as
\begin{equation}
\label{eq:kemeny-w}
K(G)
=
\operatorname{tr}(\LL_G^\#\DD)
-
\frac{1}{\operatorname{vol}(G)}
\w^\top\LL_G^\#\w.
\end{equation}

Thus, the second term in~\eqref{eq:kemeny-w}, which vanishes in the
regular case, is governed by the Rayleigh quotient of $\LL_G^\#$
evaluated at the orthogonal degree vector. Notice that
\[
\|\w\|_2^2
=
\|\k\|_2^2
-
\frac{\operatorname{vol}(G)^2}{n}.
\]
The vector $\w$ was introduced in~\cite[Eq.~(25)]{Wa17} for the case
of graphs, and its norm provides a natural measure of degree
heterogeneity.

Since $\w\in\mathbf{1}^\perp$, the Courant--Fischer min--max theorem gives
\[
\frac{1}{\gamma_n}
\leq
\frac{\w^\top\LL_G^\#\w}{\|\w\|_2^2}
\leq
\frac{1}{\gamma_2}
\]
whenever $\w\neq\mathbf{0}$. Substituting these inequalities into
\eqref{eq:kemeny-w} yields the bounds obtained in~\cite{Wa17}:
\begin{equation}
\label{eq:upper}
\begin{aligned}
K(G)
&\geq
\operatorname{tr}(\LL_G^\#\DD)
-
\frac{1}{\operatorname{vol}(G)\gamma_2}
\left(
\|\k\|_2^2-\frac{\operatorname{vol}(G)^2}{n}
\right),\\
K(G)
&\leq
\operatorname{tr}(\LL_G^\#\DD)
-
\frac{1}{\operatorname{vol}(G)\gamma_n}
\left(
\|\k\|_2^2-\frac{\operatorname{vol}(G)^2}{n}
\right).
\end{aligned}
\end{equation}

The equality cases in~\eqref{eq:upper} have a direct spectral
interpretation.

\begin{proposition}
\label{propo:rayleigh-equality}
Let $G$ be a connected non-regular network, and let
\[
\w=\k-\frac{\operatorname{vol}(G)}{n}\mathbf{1}.
\]
Then equality holds in the lower bound in~\eqref{eq:upper} if and only
if $\w$ belongs to the eigenspace of $\LL_G$ associated with
$\gamma_2$. Likewise, equality holds in the upper bound if and only if
$\w$ belongs to the eigenspace associated with $\gamma_n$.
\end{proposition}

\begin{proof}
The nonzero eigenvalues of $\LL_G^\#$ are
\[
\frac{1}{\gamma_2},\ldots,\frac{1}{\gamma_n},
\]
with the same eigenspaces as the corresponding eigenvalues of
$\LL_G$. Since $\w\in\mathbf{1}^{\perp}$, equality in the upper,
respectively lower, Rayleigh-quotient bound occurs precisely when
$\w$ belongs to the eigenspace associated with $\gamma_2$,
respectively $\gamma_n$. The conclusion follows from
\eqref{eq:kemeny-w}.
\end{proof}

Proposition~\ref{propo:rayleigh-equality} suggests considering the more
general situation in which the orthogonal degree vector belongs to an
arbitrary Laplacian eigenspace, rather than necessarily to one of the
two extremal eigenspaces. This leads naturally to the following notion.

We say that a network $G$ is \emph{slightly regular} if the subspace
spanned by $\w$ is invariant under $\LL_G$. Equivalently, either
$\w=\mathbf{0}$, in which case $G$ is regular, or $\w$ is an
eigenvector of $\LL_G$.

For a $k$-regular network, $\w=\mathbf{0}$ and
\eqref{thm:wangetal} reduces to the well-known identity
\[
K(G)=k\,\operatorname{tr}(\LL_G^\#).
\]
For non-regular slightly regular networks, the Rayleigh quotient above
is not merely bounded by its extremal values but can be evaluated
exactly.

\begin{proposition}
\label{propo:slightly}
Let $G$ be a non-regular slightly regular network, and let $\gamma$ be
the Laplacian eigenvalue satisfying
\(
\LL_G\w=\gamma\w.
\)
Then
\begin{equation}
\label{eq:slightly-exact}
K(G)
=
\operatorname{tr}(\LL_G^\#\DD)
-
\frac{\gamma^\#}{\operatorname{vol}(G)}
\left(
\|\k\|_2^2-\frac{\operatorname{vol}(G)^2}{n}
\right).
\end{equation}
Moreover,
\[
\begin{aligned}
K(G)
&\geq k_{\min}\sum_{i=2}^n\frac{1}{\gamma_i}-
\frac{\gamma^\#}{\operatorname{vol}(G)}
\left(
\|\k\|_2^2-\frac{\operatorname{vol}(G)^2}{n}
\right),\\[2mm]
K(G)
&\leq k_{\max}\sum_{i=2}^n\frac{1}{\gamma_i}-
\frac{\gamma^\#}{\operatorname{vol}(G)}
\left(
\|\k\|_2^2-\frac{\operatorname{vol}(G)^2}{n}
\right).
\end{aligned}
\]

\end{proposition}

\begin{proof}
Since $\LL_G\w=\gamma\w$ and $\w\neq\mathbf{0}$,
\(
\LL_G^\#\w=\gamma^\#\w.
\)
Therefore,
\[
\k^\top\LL_G^\#\k
=
\w^\top\LL_G^\#\w
=
\gamma^\#\|\w\|_2^2
=
\gamma^\#
\left(
\|\k\|_2^2-\frac{\operatorname{vol}(G)^2}{n}
\right),
\]
and~\eqref{eq:slightly-exact} follows from
\eqref{thm:wangetal}.

Finally, since
\(
k_{\min}I\preceq\DD\preceq k_{\max}I
\)
and $\LL_G^\#\succeq0$, we obtain
\[
k_{\min}\operatorname{tr}(\LL_G^\#)
\leq
\operatorname{tr}(\LL_G^\#\DD)
\leq
k_{\max}\operatorname{tr}(\LL_G^\#).
\]
Since
\(
\operatorname{tr}(\LL_G^\#)
=
\sum_{i=2}^n\frac{1}{\gamma_i},
\)
the stated bounds follow.
\end{proof}

The path $P_4$ provides a simple example showing that slight regularity
is not restricted to the extremal eigenspaces appearing in
Proposition~\ref{propo:rayleigh-equality}. Indeed,
\[
\w=\frac{1}{2}(-1,1,1,-1)^\top,
\qquad
\LL_{P_4}\w=2\w,
\]
whereas the largest Laplacian eigenvalue of $P_4$ is
\(
\gamma_4=2+\sqrt{2}
\) and the lowest positive eigenvalue is \(
\gamma_2=2-\sqrt{2}
\). 
Thus, in this example the eigenvalue associated with $\w$ is an
interior Laplacian eigenvalue.

We next identify a broad structural setting in which slight regularity
arises naturally. Recall that a \emph{semiregular network} is a
network with exactly two distinct vertex degrees. The hypothesis in
the following result amounts to requiring that the two degree classes
form a weighted equitable partition.

\begin{proposition}
\label{propo:semi}
Let
\(
G=(V_0\coprod V_1,E,c)
\)
be a $(k_0,k_1)$-semiregular network, where every vertex in $V_i$ has
degree $k_i$ and
\(
n_i:=|V_i|, i=0,1.
\)
Assume that, for every $i\in V_0$,
\(
\sum_{j\in V_0}c(i,j)=r_0,\) and \(
\sum_{j\in V_1}c(i,j)=k_0-r_0,
\)
whereas, for every $i\in V_1$,
\(
\sum_{j\in V_0}c(i,j)=k_1-r_1,\) and \(
\sum_{j\in V_1}c(i,j)=r_1.
\)
Then $G$ is slightly regular, with associated Laplacian eigenvalue
\[
\gamma
=
\frac{n}{n_0}(k_1-r_1)
=
\frac{n}{n_1}(k_0-r_0).
\]
In particular,
\(
\gamma=n
\)
if and only if
\(
k_0=r_0+n_1,
\)
or equivalently,
\(
k_1=r_1+n_0.
\)\end{proposition}

\begin{proof}
The orthogonal degree vector of $G$ is
\[
\w
=
\frac{k_0-k_1}{n}
\left(
n_1\chi_{V_0}-n_0\chi_{V_1}
\right).
\]
For every vertex in $V_0$,
\[
(\LL_G\w)_i
=
(k_0-k_1)(k_0-r_0),
\]
whereas, for every vertex in $V_1$,
\[
(\LL_G\w)_i
=
-(k_0-k_1)(k_1-r_1).
\]
Hence
\[
\LL_G\w=\gamma\w,
\]
where
\[
\gamma
=
\frac{n}{n_1}(k_0-r_0)
=
\frac{n}{n_0}(k_1-r_1).
\]
The equality between these two expressions follows by double counting
the total conductance between $V_0$ and $V_1$:
\[
n_0(k_0-r_0)=n_1(k_1-r_1).
\]
The last assertion follows immediately from the expression for
$\gamma$.
\end{proof}

In the unweighted case, $k_0-r_0\leq n_1$ and
$k_1-r_1\leq n_0$. Therefore, $\gamma=n$ precisely when every vertex
of $V_0$ is adjacent to every vertex of $V_1$, that is, when the cut
between the two degree classes is complete. Moreover, since the largest
Laplacian eigenvalue of an unweighted graph satisfies
$\gamma_n\leq n$ \cite{GM94}, in this case
\[
\gamma=n=\gamma_n.
\]

The following examples illustrate the scope of
Proposition~\ref{propo:semi}.

\begin{example}{\rm
\leavevmode
\begin{itemize}

\item[(a)]
Regular graphs are slightly regular by definition, since
$\w={\sf 0}$. For a $k$-regular graph we recover
\(
K(G)=k\,\operatorname{tr}(\LL_G^\#).
\)

\item[(b)]
Let $G=K_{p,q}$ be a complete bipartite graph and set $n=p+q$.
Taking
\[
n_0=q,\qquad n_1=p,\qquad
k_0=p,\qquad k_1=q,\qquad
r_0=r_1=0,
\]
Proposition~\ref{propo:semi} gives
\[
\gamma=n=\gamma_n.
\]
Using Proposition~\ref{propo:slightly} together with the expression
for $\LL_G^\#$ obtained in~\cite{AbCaEnJi23}, we obtain
\[
K(G)
=
n-1-\frac{2pq}{n^2}
-\frac{(p-q)^2}{2n^2}
=
n-\frac{3}{2}.
\]
This agrees with the value obtained directly from the spectrum of the
normalized Laplacian of $K_{p,q}$.

\end{itemize}
}\end{example}

A broader family is obtained from join constructions studied in \cite{BCEM12}. 
 For any $i=1,\ldots,s>1$, let $G_i=(W_i,E_i)$ be a connected
$k$-regular graph of order $n_i$, let $G_0=(W_0,E_0)$ be a 
$\ell$-regular graph of order $n_0$ and let $$V=\coprod\limits_{i=0}^sW_i$$ be the disjoint union of
all vertex sets. Let $n=\sum\limits_{i=0}^s n_i.$
The {\it join graph with base
$G_0$ and satellites $G_i$} is the graph  $G=(V,E)$ obtained by joining the
graphs $G_i$, $i=1,\ldots,s$ to $G_0$; that is, $E= \hat E\cup \bigcup\limits_{i=0}^s E_i,$ where 
$$\hat E=\Big\{\{x,y\}: x\in W_0 \mbox{ and } y\in \coprod\limits_{i=1}^sW_i\Big\}.$$ In this case, $V_0=W_0$, $V_1=\coprod\limits_{i=1}^sW_i$, $k_0=\ell+n-n_0$, $r_0=\ell$, $k_1=k+n_0$, $r_1=k$ and  $\gamma_n=n.$
 
If $k_0=k_1$, the graph is regular and hence slightly regular as well.
Thus, every graph in this family is slightly regular.

The windmill and generalized windmill graphs considered in
\cite{K2019} are particular instances of this construction. A windmill
graph $W(s,k)$ consists of $s$ disjoint copies of $K_k$, every
vertex of which is adjacent to a common central vertex. In the notation
above, this corresponds to
\[
G_0=K_1,
\qquad
G_i=K_k,\quad i=1,\ldots,s.
\]

Kooij~\cite{K2019} introduced two natural generalizations obtained by
replacing the central vertex with $n_0$ central vertices. In the
generalized windmill graph of Type I, denoted
\[
W'(s,k,n_0),
\]
the central vertices form a clique $K_{n_0}$, whereas in the generalized
windmill graph of Type II, denoted
\[
W''(s,k,n_0),
\]
the central vertices are pairwise non-adjacent. Thus, in the
construction above, the base graph $G_0$ is respectively
$K_{n_0}$ or $\overline{K}_{n_0}$, while the satellites are copies of
$K_k$. Consequently, all these graphs are slightly regular.

For every non-regular member of these windmill families,
Proposition~\ref{propo:semi} gives
\[
\gamma=n=\gamma_n.
\]
Hence, by Proposition~\ref{propo:rayleigh-equality}, the upper bound in
\eqref{eq:upper} is attained. This provides a common spectral
explanation for the equality cases established for windmill and
generalized windmill graphs in~\cite[Proposition~1]{KoDu20}. In this
sense, slight regularity is not restricted to those particular
families: it identifies the spectral mechanism underlying their
equality cases, while Proposition~\ref{propo:slightly} extends the same
mechanism to arbitrary Laplacian eigenspaces, as illustrated by $P_4$.

The variational viewpoint developed in this section will also be useful
in the approximation setting considered next. Here, exact alignment of
the orthogonal degree vector with a Laplacian eigenspace turns a
Rayleigh-quotient bound into an identity. In the next section, spectral
sparsification provides multiplicative control of Laplacian quadratic
forms, and Rayleigh-quotient arguments will instead be used to obtain
approximate spectral control and, ultimately, bounds for Kemeny's
constant.

\section{Approximating Kemeny's constant via spectral sparsification}
\label{sec:sparsification}

Although the results in Section~\ref{sec:slightly} yield an exact
expression for Kemeny's constant for slightly regular networks, its
direct evaluation still involves quantities associated with the group
inverse of the Laplacian and, using standard dense linear algebra,
has cubic computational cost in general. This motivates the search for
efficient approximation methods that apply to arbitrary networks.

The approach developed in this section is closely related to the
variational viewpoint used in the previous section. There,
Rayleigh-quotient arguments were used to exploit exact spectral
information on the orthogonal degree vector. Here, instead of requiring
such an exact spectral relation, we study how Kemeny's constant behaves
when the quadratic form of the Laplacian is preserved only
approximately. This leads naturally to spectral sparsification.

More precisely, we derive multiplicative bounds relating the Kemeny
constant of a network to that of a spectral approximation. We then use
the sparsity of the resulting network to approximate Kemeny's constant
through the solution of Laplacian linear systems, analyse the
computational cost of the resulting algorithm, and compare its
theoretical guarantees and complexity with those of the approximation
method introduced in~\cite{XSZKZ20}.

Spectral sparsification, introduced by Spielman and Teng
in~\cite{ST2004}, aims to replace a network by a much sparser weighted
network while approximately preserving the quadratic form associated
with its Laplacian. This allows spectral and linear-algebraic
computations on the original network to be transferred to a
substantially sparser one. Throughout this section, we adopt the notion
of spectral approximation used in~\cite{SpSr2011}.

For $0<\epsilon <1$, we say that a  network $G'=(V',E',c')$ is an $\epsilon$-\emph{approximation
 of a network} $G=(V,E,c)$ if $V'=V$, $E'\subseteq E$ and for all ${\sf z}
\in \mathbb{R}^n$ it holds that
\begin{equation} \label{dispersa}
(1-\epsilon)\z^\top \LL_{G}\z \leq \z^\top \LL_{G'}\z \leq(1+\epsilon) \z^\top \LL_{G}\z.
\end{equation}
Since $G$ is connected and $0<\epsilon<1$, any
$\epsilon$-approximation $G'$ is connected as well.
Observe that $(V,E')$ is a subgraph of $(V,E)$, with  the same vertex set and satisfying \eqref{dispersa}.

The sparsifier need not be a proper subgraph. For example, if
$E'=E$ and
\(
(1-\epsilon)c(e)\leq c'(e)\leq(1+\epsilon)c(e),\,\, e\in E,
\)
then $G'=(V,E,c')$ is an $\epsilon$-approximation of $G$.
We denote the degree vector of \(G'\) by \(\k' = (k_1', \ldots, k_n')^\top\), and we denote its diagonal degree matrix by
\[\DD' = \operatorname{diag}(k_1', \ldots, k_n').
\]

Given any $\epsilon$-approximation $G'$ of $G$,   the Courant-Fischer  min--max theorem   implies that  $\LL_G$ and $\LL_{G'}$ have approximately the same eigenvalues, specifically
$$(1-\epsilon)\gamma_i\le \gamma'_i\le (1+\epsilon)\gamma_i, \hspace{.25cm} i=1,\ldots,n,$$
see \cite{SpSr2011}. In this paper, Spielman and Srivastava  state that the same multiplicative bounds also hold for the eigenvalues of the normalized Laplacian. Since the degree matrix changes under sparsification, this conclusion does not hold in general, as the following example shows: Given $a,b>0$, consider $G=G(a,b)$, the path on $4$ vertices whose conductances are $c_{12}=c_{34}=a$, $c_{23}=b$ and $0$ otherwise. Then, the   eigenvalues of its normalized Laplacian are $\big\{0,\frac{b}{a+b}, \frac{2a+b}{a+b},2\big\}$.

For $0<\epsilon<1$, we consider $G'=G\big(a(1+\epsilon),b(1-\epsilon)\big)$ and $G''=G\big(a(1-\epsilon),b(1+\epsilon)\big)$ and we denote by $\mu_2$, $\mu_2'$ and $\mu_2''$ the second eigenvalues of the    normalized Laplacian of $G$, $G'$ and $G''$, respectively. Then, $G'$ and $G''$ are $\epsilon$-approximations of $G$ and 
$$\dfrac{\mu_2'}{\mu_2}=\frac{(1-\epsilon)(a+b)}{a(1+\epsilon)+b(1-\epsilon)}< 1+\epsilon\hspace{.25cm}\hbox{and}\hspace{.25cm}\dfrac{\mu_2''}{\mu_2}=\frac{(1+\epsilon)(a+b)}{a(1-\epsilon)+b(1+\epsilon)}>1-\epsilon.$$
However,  $1-\epsilon\le \dfrac{\mu_2'}{\mu_2}$ and $\dfrac{\mu_2''}{\mu_2}\le 1+\epsilon$ if and only if $a\le b$ and hence, 
when $0<b<a$ we have that $\dfrac{\mu_2'}{\mu_2}< 1-\epsilon$ and $\dfrac{\mu_2''}{\mu_2}>1+\epsilon$.

Nevertheless,  we have the following relationship between the eigenvalues of the normalized Laplacian of a network and those of the normalized Laplacian of any approximation.

\begin{proposition}
    \label{ineq}
 Given  $0<\epsilon < 1$, for any $\epsilon$-approximation $G'$ of a network $G$ we have that 
          $$ \sigma_\epsilon^{-1}\mu_i\leq \mu_i'\leq \sigma_\epsilon\mu_i,\hspace{.25cm}i=2,\ldots,n      $$ 
          where  $\sigma_\epsilon=\frac{1+\epsilon}{1-\epsilon}$.  
\end{proposition}
\proof
Taking $\z={\sf e}_i$ in (\ref{dispersa}), where ${\sf e}_i$ denotes the $i$-th vector of the canonical basis, yields the following relationship:
\[
(1-\epsilon)k_i
\leq k_i'
\leq (1+\epsilon)k_i,
\]
for every $i=1,\ldots,n$. Therefore, for all $\z\in\mathbb{R}^n$,
\[
(1-\epsilon)\z^\top \DD\z \leq \z^\top \DD'\z \leq(1+\epsilon) \z^\top \DD\z,
\]
where $\DD$ and $\DD'$ are the diagonal degree matrices of $G$ and $G'$, respectively.
Let
\[
\mathcal{R}_G(\z)
=
\frac{\z^{\top}\LL_G\z}{\z^{\top}\DD\z},
\qquad
\mathcal{R}_{G'}(\z)
=
\frac{\z^{\top}\LL_{G'}\z}{\z^{\top}\DD'\z},
\]
for every nonzero vector $\z$.

From the spectral approximation of the Laplacians and the comparison between the degree matrices, we obtain
\[
\sigma_\epsilon^{-1} \mathcal{R}_G(\z)
\leq
\mathcal{R}_{G'}(\z)
\leq
\sigma_\epsilon \mathcal{R}_G(\z)
\]
for every nonzero vector $\z$. The eigenvalues of the normalized Laplacian of $G$ are the
generalized eigenvalues of the pair $(\LL_G,\DD)$. Hence,
\[
\mu_i
=
\min_{\substack{S\subseteq\mathbb{R}^n\\ \dim S=i}}
\;
\max_{\substack{\z\in S\\ \z\neq 0}}
\mathcal{R}_G(\z),
\qquad i=1,\ldots,n,
\]
and the analogous formula holds for $G'$. Applying the generalized
Courant--Fischer min--max theorem to the pointwise inequalities above
gives
\[
\sigma_\epsilon^{-1}\mu_i
\leq
\mu_i'
\leq
\sigma_\epsilon\mu_i,
\qquad i=1,\ldots,n.
\]

\qed

The constants in Proposition~\ref{ineq} are sharp. Indeed, in the
preceding path example, when $b\to 0$ we have that 
$\dfrac{\mu_2'}{\mu_2}\to \sigma_\epsilon^{-1}$ whereas $\dfrac{\mu_2''}{\mu_2}\to \sigma_\epsilon$.

Proposition~\ref{ineq} immediately yields a corresponding multiplicative
stability result for Kemeny's constant.

\begin{corollary}\label{propo:direct}
Let $0<\epsilon<1$ and let $G'$ be an $\epsilon$-approximation of
$G$. Then
\[
\sigma_\epsilon^{-1}K(G)
\leq
K(G')
\leq
\sigma_\epsilon K(G).
\]
\end{corollary}

\begin{proof}
By Proposition~\ref{ineq},
\[
\sigma_\epsilon^{-1}\mu_i
\leq
\mu_i'
\leq
\sigma_\epsilon\mu_i,
\qquad i=2,\ldots,n.
\]
Since all these eigenvalues are positive,
\[
\frac{\sigma_\epsilon^{-1}}{\mu_i}
\leq
\frac{1}{\mu_i'}
\leq
\frac{\sigma_\epsilon}{{\mu_i}}.
\]
Summing over $i=2,\ldots,n$ and using~\eqref{Kemeny:eign} gives the result.
\end{proof}

Multiple algorithms to compute a sparse $\epsilon$-approximation of a network have been proposed in the literature. In \cite{SpSr2011}, Spielman and Srivastava proposed an algorithm which computes an $\epsilon$-approximation with at most $\mathcal{O}\left(\epsilon^{-2}n\log{n}\right)$ edges of any network, with high probability. The algorithm is based on a random sampling of the edges of a network with probability proportional to their conductance times an approximation of their effective resistance. In \cite{KLP12}, this algorithm was modified, decreasing its computational cost to $\widetilde{\mathcal{O}}\left(m\log^2{n}\log(\epsilon^{-1})\right)$\footnote{We use the $\widetilde{\mathcal{O}}$ notation to hide $\log{\log{n}}$ factors.}. 

Also in \cite{KLP12}, the authors proposed another algorithm which computes an $\epsilon$-approximation with at most $\mathcal{O}\left(\epsilon^{-2}n\log{n}\right)$ edges of any network with $m>n\log^5n$. The computational cost of this algorithm is $$\widetilde{\mathcal{O}}\left(m\log_{m/(n\log^5{n})}{n}\log(\epsilon^{-1})\right),$$ and in the range where $m>n^{r+1}$
for some constant $r>0$, the computational cost becomes $$\widetilde{\mathcal{O}}\left(m\log(\epsilon^{-1})\right).$$

Corollary~\ref{propo:direct} shows that, for small \(\epsilon\), the Kemeny constant of any \\ \(\epsilon\)-approximation \(G'\) of \(G\) is itself a good approximation of the Kemeny constant of \(G\).
However, to avoid the \(\mathcal{O}(n^3)\) cost of computing the full spectrum, we must approximate \(K(G')\) using a procedure that takes advantage of the sparsity of \(G'\).

Following the approach in~\cite{XSZKZ20}, we rewrite~\eqref{Kemeny:eign} to express the Kemeny constant in terms of evaluations of the quadratic form associated with the pseudoinverse of the Laplacian of \(G'\).
This is convenient because the Laplacian is a symmetric diagonally dominant (SDD) matrix, and approximate solutions to linear systems with sparse SDD matrices can be obtained efficiently.

For a positive semidefinite matrix $\MM$, we denote by
\[
\|\x\|_{\MM}:=(\x^\top\MM\x)^{1/2}
\]
the associated energy seminorm. Given a consistent system
\[
\MM\x=\b
\]
and $0<\beta<1$, we say that $\y$ is a $\beta$-approximate solution if
\[
\|\y-\MM^\#\b\|_{\MM}
\leq
\beta\,\|\MM^\#\b\|_{\MM}.
\]
In~\cite{KLPSS16}, the Sparsified Multigrid algorithm was introduced.
For an SDD matrix $\MM$ of order $n$ with $s$ nonzero entries, the
algorithm constructs a solver in $$\mathcal{O}\left(s\log{n}+n\log^{2+o(1)}{n}\right)$$ time. This solver can be used to obtain a $\beta$-approximate solution to any system of equations in $\MM$ in $$\mathcal{O}\left(\left(s+n\log^{1+o(1)}{n}\right)\log(\beta^{-1})\right)$$ time.

 We propose Algorithm~\ref{sparsifiedkemeny} to approximate \(K(G)\). The algorithms used to compute the $\epsilon$-approximation of $G$ and the $\beta$-approximate solutions to the Laplacian systems are not specified in the general version of our algorithm. In Theorem~\ref{theor:coste}, we derive, to the best of our knowledge,
the best currently available computational complexity bound for
Algorithm~\ref{sparsifiedkemeny}, obtained by combining existing
methods for spectral sparsification and the solution of Laplacian
linear systems.

Let $G'$ be an $\epsilon$-approximation of $G$, with degree vector
$\k'=(k'_1,\ldots,k'_n)^\top$ and diagonal degree matrix
$\DD'=\operatorname{diag}(k'_1,\ldots,k'_n)$.

For each $i=1,\ldots,n$, define
\[
\z_i
=
\sqrt{k'_i}
\left(
\e_i-\frac{\k'}{\operatorname{vol}(G')}
\right).
\]
Notice that
\(
\mathbf{1}^\top\z_i=0,
\)
so the Laplacian system
\(
\LL_{G'}\x=\z_i
\)
is consistent.

Using the relation between the group inverses of the normalized and
combinatorial Laplacians, we obtain
\[
K(G')
=
\sum_{i=1}^n
\z_i^\top\LL_{G'}^\#\z_i.
\]
Thus, computing $K(G')$ reduces to evaluating $n$ quadratic forms
associated with $\LL_{G'}^\#$.

If $\y_i$ is a $\beta$-approximate solution of
\[
\LL_{G'}\x=\z_i,
\]
then
\[
\z_i^\top\LL_{G'}^\#\z_i
-
\left(
2\y_i^\top\z_i
-
\y_i^\top\LL_{G'}\y_i
\right)
=
\|\y_i-\LL_{G'}^\#\z_i\|_{\LL_{G'}}^2.
\]
Consequently, the quantity in parentheses provides a lower
approximation to the corresponding quadratic form, with an error
controlled by $\beta$.

 {\normalsize{	\begin{algorithm}
     {\normalsize{	\caption{: $K_\beta(G')= \SparsifiedKemeny (G,\epsilon, \beta)$.}\label{sparsifiedkemeny}
	\begin{algorithmic}		
        \State{{\bf Input}: A network $G$ and real numbers $0<\epsilon < 1$, $0<\beta<1$.}
        \State{Set $K_\beta(G'):=0$.}
		\State{Compute $G'$, an $\epsilon$-approximation of $G$.}		
		\For{$i = 1 \, \, \text{to} \, \, n$}	
               \State{Set $\z_i:=\sqrt{k_i'}\left({\sf e}_i-\frac{1}{{\rm vol}(G')}\k'\right)$.}
               		\State{Compute $\y_i$, a $\beta$-approximate solution to $\LL_{G'}\x=\z_i$.}
        \State{Set $K_\beta(G'):=K_\beta(G')+ 2\y_i^\top \z_i-\y_i^\top \LL_{G'}\y_i$.}	
		\EndFor	
		\State{{\bf Return $K_\beta(G')$}.}
	\end{algorithmic}}}
\end{algorithm}}}

In the following result, we show an upper and lower bound for our approximation of Kemeny's constant, denoted by $K_\beta(G')$, in terms of the exact Kemeny constant.
 This result quantifies the two sources of error in
Algorithm~\ref{sparsifiedkemeny}: the spectral approximation error,
controlled by $\epsilon$, and the error in the Laplacian solves,
controlled by $\beta$.

\begin{theorem}\label{theor:aprox}
    Let $0<\epsilon < 1$,  $0<\beta< 1$, and let $G$ be a network. Then the output of Algorithm \ref{sparsifiedkemeny} satisfies 
    \[\sigma_\epsilon^{-1}(1-\beta^2) K(G)\leq K_\beta(G')\leq \sigma_\epsilon K(G),\]
where
\[
\sigma_\epsilon=\frac{1+\epsilon}{1-\epsilon}.
\]

\end{theorem}

\begin{proof}
In \cite{B13}, the following expression for the group inverse of the normalized Laplacian matrix was obtained
{\small{
\begin{equation}\label{pseudol}
    \L_{G'}^\#=\left(\II-\frac{1}{{\rm vol}(G')}\DD'^{1/2} \1 \1^\top \DD'^{1/2}\right)\DD'^{1/2} \LL_{G'}^\# \DD'^{1/2}\left(\II-\frac{1}{{\rm vol}(G')}\DD'^{1/2} \1 \1^\top \DD'^{1/2}\right),
\end{equation}
}}
where $\1$ denotes the vector whose entries are all equal to $1$. For each $i=1,...,n$ we define
$$\z_i=\DD'^{1/2}\left(\II-\frac{1}{{\rm vol}(G')}\DD'^{1/2} \1 \1^\top \DD'^{1/2}\right) {\sf e}_i = \sqrt{k_i'}\left({\sf e}_i-\frac{1}{{\rm vol}(G')}\k'\right).$$ Then, from (\ref{Kemeny:eign}) and (\ref{pseudol}) we obtain
$$K(G')=\sum_{i=1}^{n}{{\sf e}_i^\top \L_{G'}^\# {\sf e}_i}=\sum_{i=1}^{n}{\z_i^\top \LL_{G'}^\# \z_i}.$$

Now, for each $i=1,...,n$, let $\y_i$ be a $\beta$-approximate solution to $\LL_{G'}\x=\z_i$, and let ${\sf h}_i=2\y_i^\top \z_i-\y_i^\top \LL_{G'}\y_i$. The vector $\y_i$ satisfies $\|\y_i-\LL_{G'}^\#\z_i\|_{\LL_{G'}}^2\leq \beta^2\|\LL_{G'}^\#\z_i\|_{\LL_{G'}}^2$. Now, on the one hand, by the definition of group inverse,
$$\|\LL_{G'}^\#\z_i\|_{\LL_{G'}}^2= \z_i^\top \LL_{G'}^\# \z_i.$$
And, on the other hand,
$$\|\y_i-\LL_{G'}^\#\z_i\|_{\LL_{G'}}^2 = \y_i^\top \LL_{G'}\y_i - 2\y_i^\top \LL_{G'} \LL_{G'}^\#\z_i + \z_i^\top \LL_{G'}^\# \LL_{G'} \LL_{G'}^\# \z_i,$$
and considering that $\z_i$ is orthogonal to $\1$, we have
$$\|\y_i-\LL_{G'}^\#\z_i\|_{\LL_{G'}}^2 = \z_i^\top \LL_{G'}^\# \z_i - {\sf h}_i.$$

Therefore,
$$0\leq K(G')-K_\beta(G')=\sum_{i=1}^{n}{\left(\z_i^\top \LL_{G'}^\# \z_i-{\sf h}_i\right)}\leq \beta^2 \sum_{i=1}^{n}{\z_i^\top \LL_{G'}^\# \z_i}=\beta^2 K(G').$$

Finally, using the previous bounds and Corollary \ref{propo:direct}, we obtain the desired result
\[\sigma_\epsilon^{-1}(1-\beta^2) K(G) \leq (1-\beta^2)K(G') 
\leq K_\beta(G') \leq  K(G')\leq \sigma_\epsilon K(G).
\qedhere\]
\end{proof}

Notice that the approximation error due to the linear solver is
one-sided:
\[
K_\beta(G')\leq K(G').
\]

\begin{theorem}\label{theor:coste}
For any network, it is possible to compute Algorithm \ref{sparsifiedkemeny} in $$\mathcal{O}\left(n^2\log{n}\log{\beta^{-1}} \left(\epsilon^{-2} + \log^{o(1)}{n}\right)\right)$$
time. Therefore, for fixed values of $\epsilon$ and $\beta$, the computational cost of Algorithm \ref{sparsifiedkemeny} is $\mathcal{O}\left(n^2\log^{1+o(1)}{n}\right)$, that is, it is nearly quadratic in the number of vertices.

\end{theorem}

\begin{proof}
We fix a constant $0<r<1$. As stated before, in the case when $m>n\log^5n$ and $m>n^{r+1}$, we can compute the $\epsilon$-approximation $G'$ in $\widetilde{\mathcal{O}}\left(m\log(\epsilon^{-1})\right)$ time with the algorithm introduced in \cite{KLP12}. In the opposite case, we can use the other algorithm of \cite{KLP12} to obtain $G'$ in $\widetilde{\mathcal{O}}\left(m\log^2{n}\log(\epsilon^{-1})\right)$ time. 

We denote by $m'$ the number of edges of $G'$. The computational cost of obtaining the solver for $\LL_{G'}$ using the algorithm Sparsified Multigrid from \cite{KLPSS16} is 
$$\mathcal{O}\left(m'\log{n}+n\log^{2+o(1)}{n}\right)$$ time. This time is asymptotically smaller in all variables than the cost of obtaining $\beta$-approximate solutions to the $n$ systems involving this matrix, which is $$\mathcal{O}\left(\left(nm'+n^2\log^{1+o(1)}{n}\right)\log(\beta^{-1})\right).$$

Regardless of the algorithm used to compute $G'$, $m'$ is at most $\mathcal{O}\left(\epsilon^{-2}n\log{n}\right)$. Therefore, the cost of obtaining $\beta$-approximate solutions to the $n$ systems is $$\mathcal{O}\left(n^2\log{n}\log{\beta^{-1}} \left(\epsilon^{-2} + \log^{o(1)}{n}\right)\right).$$ This is asymptotically larger than the time needed to compute $G'$, because in the worst-case scenario, when $m$ is of order $n^2$, $G'$ can be computed in $\widetilde{\mathcal{O}}\left(n^2\log(\epsilon^{-1})\right)$ time.

It remains to account for the cost of constructing the vectors
$\z_i$ and evaluating the $n$ quantities
$
2\y_i^\top\z_i-\y_i^\top\LL_{G'}\y_i.
$
Each vector $\z_i$ can be constructed in $\mathcal{O}(n)$ time,
whereas the two terms $2\y_i^\top\z_i$ and
$\y_i^\top\LL_{G'}\y_i$ can be evaluated in
$\mathcal{O}(n)$ and $\mathcal{O}(m')$ time, respectively. Hence, the total additional
cost is
\[
\mathcal{O}(n^2+nm').
\]
This cost is also asymptotically dominated by the
complexity of solving the $n$ Laplacian systems and therefore does
not change the stated complexity bound.
\end{proof}
In~\cite{KLPSS16}, the authors also introduce the
\emph{Sparsified Cholesky} algorithm. After a polynomial-time
preprocessing stage, the resulting solver can be applied to an SDD
system with a cost that is linear in the number of nonzero entries of
the matrix. However, as discussed in~\cite{KLPSS16}, a nearly
linear-time construction of such a solver is not currently available.
If a solver with both nearly linear preprocessing cost and linear solve
cost were available, it could be incorporated into our framework. Since
the spectral approximation has
\[
O(\epsilon^{-2}n\log n)
\]
edges, this would reduce the computational cost of our approximation to
\[
O(n^2\log n)
\]
for fixed values of $\epsilon$ and $\beta$.

We next compare the computational complexity and approximation
guarantees of our method with the approximation proposed
in~\cite{XSZKZ20}. The approach in~\cite{XSZKZ20} also uses the
representation~\eqref{pseudol} and approximate solutions of SDD linear
systems. However, instead of first sparsifying the network, the authors
apply Hutchinson's randomized trace estimator to reduce the number of
linear systems that must be solved for the original network.

More precisely, their method computes $\beta$-approximate solutions to
\[
48\epsilon^{-2}\log(2n)
\]
linear systems, with
\[
\beta=\frac{\epsilon n^{-2.5}}{2\sqrt{3}},
\]
using the SDD solver introduced in~\cite{KyngSachdeva2016}. Since the
cost of solving each system is
\[
O\!\left(
m\log^3 n\log(\beta^{-1})
\right),
\]
the resulting computational complexity is
\[
O\!\left(
m\epsilon^{-2}\log^4 n
\bigl(\log n+\log(\epsilon^{-1})\bigr)
\right).
\]
For dense networks, where $m=\Theta(n^2)$, this becomes
\[
O\!\left(
n^2\epsilon^{-2}\log^4 n
\bigl(\log n+\log(\epsilon^{-1})\bigr)
\right),
\]
and hence, for fixed $\epsilon$,
\[
O(n^2\log^5 n)
\]
with respect to the number of vertices.

For comparison, suppose that the SDD solver used in the framework
of~\cite{XSZKZ20} were replaced by Sparsified Multigrid. Since their
method requires $O(\epsilon^{-2}\log n)$ Laplacian solves, and each
solve would have cost
\[
O\!\left(
\left(m+n\log^{1+o(1)}n\right)
\bigl(\log n+\log(\epsilon^{-1})\bigr)
\right),
\]
the resulting complexity would be
\[
O\!\left(
\epsilon^{-2}\log n
\left(m+n\log^{1+o(1)}n\right)
\bigl(\log n+\log(\epsilon^{-1})\bigr)
\right).
\]
For dense networks, this becomes
\[
O\!\left(
n^2\epsilon^{-2}\log n
\bigl(\log n+\log(\epsilon^{-1})\bigr)
\right),
\]
and therefore
\[
O(n^2\log^2 n)
\]
with respect to $n$ when $\epsilon$ is fixed. Thus, even under this
hypothetical replacement of the linear solver, the dependence on $n$
for dense networks remains larger than the
$O(n^2\log^{1+o(1)}n)$ complexity obtained by our method for fixed
approximation parameters.

We now compare the approximation guarantees. In
\cite[Theorem~8.5]{XSZKZ20}, the authors prove that their approximation
$\widetilde{K}$ satisfies
\[
(1-\epsilon)^3K(G)
\leq
\widetilde{K}
\leq
(1+\epsilon)^3K(G).
\]
On the other hand, Theorem~\ref{theor:aprox} gives
\[
\sigma_\epsilon^{-1}(1-\beta^2)K(G)
\leq
K_\beta(G')
\leq
\sigma_\epsilon K(G),
\qquad
\sigma_\epsilon=\frac{1+\epsilon}{1-\epsilon}.
\]

For the same value of $\epsilon$, the upper multiplicative factor in
Theorem~\ref{theor:aprox} is smaller than the upper factor in
\cite[Theorem~8.5]{XSZKZ20} precisely when
\[
\frac{1+\epsilon}{1-\epsilon}
<
(1+\epsilon)^3,
\]
which is equivalent to
\[
0<\epsilon<
\frac{\sqrt{5}-1}{2}
\approx 0.618.
\]
Similarly, the lower multiplicative factor in
Theorem~\ref{theor:aprox} is larger than $(1-\epsilon)^3$ precisely
when
\[
\frac{1-\epsilon}{1+\epsilon}(1-\beta^2)
>
(1-\epsilon)^3,
\]
or equivalently,
\[
\beta
<
\sqrt{
\epsilon\left(1+\epsilon-\epsilon^2\right)
}.
\]

In particular, if we choose $\beta=\epsilon$, the lower bound in
Theorem~\ref{theor:aprox} is sharper for every
$0<\epsilon<1$, whereas the upper bound is sharper whenever
\[
0<\epsilon<
\frac{\sqrt{5}-1}{2}.
\]
Consequently, in this parameter range, both multiplicative bounds
provided by Theorem~\ref{theor:aprox} are tighter than those in
\cite[Theorem~8.5]{XSZKZ20}. Moreover, for fixed $n$ and
$\beta=\epsilon$, both approaches have computational cost
$
O\!\left(
\epsilon^{-2}\log(\epsilon^{-1})
\right)
$
with respect to the approximation parameter $\epsilon$, up to factors
that are independent of $\epsilon$.

An additional feature of our approach is that the two sources of error
are controlled independently. The parameter $\epsilon$ governs the
spectral sparsification error, whereas $\beta$ controls the accuracy of
the Laplacian solves. Since the computational cost of our algorithm
depends only logarithmically on $\beta^{-1}$, $\beta$ can be chosen
substantially smaller than $\epsilon$, making the factor $1-\beta^2$ in the lower bound of
Theorem~\ref{theor:aprox} very close to one, at a relatively modest additional
cost.

\subsection{Computational results}\label{sec:comp}

We now assess the numerical performance of Algorithm~\ref{sparsifiedkemeny}
on several classes of networks. The purpose of these experiments is
twofold: to examine how much sparsification can be achieved while
preserving Kemeny's constant accurately, and to compare the observed
errors with the theoretical guarantees established above.

The complexity bound in Theorem~\ref{theor:coste} is obtained using
the spectral sparsification and Laplacian-solving algorithms described
in~\cite{KLP12,KLPSS16}. However, practical implementations of
Sparsified Multigrid are not currently available~\cite{GKS23}. For the numerical
experiments, we therefore use efficient routines from the Julia package
\texttt{Laplacians.jl}~\cite{LaplaciansJL}. The numerical implementation used
in this section should therefore be regarded as a practical realization
of Algorithm~\ref{sparsifiedkemeny}, rather than as an implementation of the
specific theoretical solver used in the complexity analysis.

In particular, Laplacian linear systems are solved using the
Approximate Cholesky method, which constructs an approximate Cholesky
factorization and uses it as a preconditioner for the conjugate
gradient method; see~\cite{GKS23}. 

For each system
\[
\LL_{G'}\x=\z_i,
\]
the iterative solver is run with Euclidean relative residual tolerance
\[
\frac{\|\LL_{G'}\y_i-\z_i\|_2}{\|\z_i\|_2}
\leq 10^{-8}.
\]
This stopping criterion is expressed using the Euclidean norm instead of the Laplacian energy seminorm, and therefore does not
directly provide the parameter $\beta$ appearing in Algorithm \ref{sparsifiedkemeny}.

To quantify the actual accuracy of the computed solutions, for each
$i$ we also compute a high-accuracy reference solution $\x_i$ of
\[
\LL_{G'}\x=\z_i
\]
by solving a reduced nonsingular Laplacian system using a Cholesky
factorization. Since $\z_i\perp\mathbf{1}$, any two solutions of this
system differ by an additive multiple of $\mathbf{1}$ and are therefore
equivalent in the $\LL_{G'}$-energy seminorm. We define the observed
solver accuracy by
\begin{equation}
\label{eq:betaobs}
\beta_{\mathrm{obs}}
=
\max_{1\leq i\leq n}
\frac{\|\y_i-\x_i\|_{\LL_{G'}}}
     {\|\x_i\|_{\LL_{G'}}}.
\end{equation}
Thus, $\beta_{\mathrm{obs}}$ is the smallest value, up to numerical
accuracy, for which all computed solutions are
$\beta$-approximations of their corresponding Laplacian systems.

To construct the spectral approximation $G'$, we use the
Spielman--Srivastava sparsification routine available in
\texttt{Laplacians.jl}. The routine takes an input sparsification
parameter, denoted by $\epsilon_{\mathrm{in}}$, but the resulting
network can in practice satisfy a substantially stronger spectral
approximation. We therefore report in the tables the observed value
$\epsilon_{\mathrm{obs}}$, defined as the smallest $\epsilon$ for which
$G'$ is an $\epsilon$-approximation of $G$.

To compute this value, let $\eta_{\min}$ and $\eta_{\max}$ denote the
smallest and largest generalized eigenvalues of the pair
\[
(\LL_{G'}^{\mathrm{red}},\LL_G^{\mathrm{red}}),
\]
where the same row and column are removed from both Laplacians. Then
\begin{equation}
\label{eq:epsilonobs}
\epsilon_{\mathrm{obs}}
=
\max\left\{
0,\,
1-\eta_{\min},\,
\eta_{\max}-1
\right\}.
\end{equation}
The generalized eigenvalues are computed using
\texttt{Arpack.eigs} with tolerance $10^{-8}$~\cite{ArpackJL}.

For validation purposes, we also compute reference values of
$K(G)$ and $K(G')$. For $H\in\{G,G'\}$, we use
\[
K(H)
=
\sum_{i=1}^n
(\z_i^H)^\top \LL_H^\#\z_i^H,
\qquad
\z_i^H
=
\sqrt{k_i^H}
\left(
\e_i-\frac{\k^H}{\operatorname{vol}(H)}
\right),
\]
evaluating the corresponding quadratic forms through high-accuracy
Cholesky solves of reduced Laplacian systems.

For each experiment, we report
$K(G')$, $K_{\beta}(G')$, the observed values
$\epsilon_{\mathrm{obs}}$ and $\beta_{\mathrm{obs}}$, the number of
edges retained after sparsification, and the relative error
\[
\frac{|K_{\beta}(G')-K(G)|}{K(G)}.
\]
The reference Cholesky solves and the generalized-eigenvalue
computations used to determine
$\beta_{\mathrm{obs}}$, $\epsilon_{\mathrm{obs}}$, $K(G)$, and
$K(G')$ are used only for validation and are not part of
Algorithm~\ref{sparsifiedkemeny}.

\subsubsection{Erd\H{o}s--R\'enyi graphs}\label{sec:comp-er}

We first consider the Erd\H{o}s--R\'enyi independent-edge model
$G(n,p)$, in which each edge is included independently with probability
$p\in(0,1)$~\cite{Gilbert59}. The expected number of edges is
\(
p{n\choose 2},
\)
so for fixed $p$ the expected number of edges grows quadratically with
the number of vertices. Previous results show that, under suitable
assumptions, Kemeny's constant for Erd\H{o}s--R\'enyi graphs exhibits
linear asymptotic growth with the number of vertices; see
\cite{OtSt23,Sy2021}.

We sample one realization of $G(1400,0.8)$. The resulting graph has
$782516$ edges and
\[
K(G)=1398.2519.
\]
The results obtained by applying Algorithm~\ref{sparsifiedkemeny} with
different sparsification levels are shown in
Table~\ref{tab:comp-er}. As explained above, the values of
$\epsilon_{\mathrm{obs}}$ and $\beta_{\mathrm{obs}}$ reported in the
table correspond to the observed spectral approximation and solver
accuracies, respectively.

\begin{table}[htbp]
\centering
\begingroup
\scriptsize
\setlength{\tabcolsep}{3.5pt}
\renewcommand{\arraystretch}{1.25}
\begin{tabular}{@{}cccccccc@{}}
\hline
$\epsilon_{\mathrm{obs}}$ & $\beta_{\mathrm{obs}}$ & $K(G')$ & $K_{\beta}(G')$ &
\shortstack{Relative error\\for $K_{\beta}(G')$} &
\shortstack{Edges\\of $G'$} &
\shortstack{Edges\\removed} &
\shortstack{\% $|E|$\\removed} \\
\hline
0.227 & $1.014\cdot 10^{-8}$ & $1403.503$ & $1403.503$ & 0.003755 & 162691 & 619825 & 79.21 \\
0.498 & $1.006\cdot 10^{-8}$ & $1414.078$ & $1414.078$ & 0.01132 & 62742 & 719774 & 91.98 \\
0.570 & $9.957\cdot 10^{-9}$ & $1423.953$ & $1423.953$ & 0.01838 & 40281 & 742235 & 94.85 \\
0.828 & $1.128\cdot 10^{-8}$ & $1450.749$ & $1450.749$ & 0.03754 & 20916 & 761600 & 97.33 \\
0.947 & $1.102\cdot 10^{-8}$ & $1479.841$ & $1479.841$ & 0.05835 & 14062 & 768454 & 98.20 \\
\hline
\end{tabular}
\endgroup
\caption{Computational results for different sparsification levels for one realization of the Erd\H{o}s--R\'enyi model $G(1400,0.8)$.}
\label{tab:comp-er}
\end{table}

The results show that, for this dense Erd\H{o}s--R\'enyi instance,
substantial sparsification is possible while retaining an accurate
approximation of Kemeny's constant. For example, when
$\epsilon_{\mathrm{obs}}=0.227$, approximately $79.2\%$ of the edges
are removed, whereas the relative error in $K_{\beta}(G')$ is only
$3.755\cdot10^{-3}$, that is, approximately $0.376\%$.
Even when more than $94\%$ of the edges are removed, the observed
relative error remains below $2\%$.

It is also worth noting that, in all the experiments reported in
Table~\ref{tab:comp-er}, $\beta_{\mathrm{obs}}$ is of order $10^{-8}$. Accordingly, $K_{\beta}(G')$ and $K(G')$ agree to the
displayed precision. Thus, in these experiments, the dominant
contribution to the approximation error comes from the spectral
sparsification step rather than from the approximate solution of the
Laplacian systems. This is also observed in all the experiments reported
below.

These results illustrate the potential of spectral sparsification for
approximating Kemeny's constant in dense networks. 

\subsubsection{Barbell graphs}\label{sec:comp-barbell}

A barbell graph consists of two disjoint complete graphs
$K_a$ and $K_b$, with $a,b\geq 2$, joined by a path
$P_\ell$, with $\ell\geq 2$, whose endpoints are identified
with one vertex of each clique. 

The graph has \[
n=a+b+\ell-2,
\qquad
m=\binom{a}{2}+\binom{b}{2}+\ell-1.
\]
 A barbell graph is called \emph{symmetric} when the
two cliques have the same order, that is, when $a=b$.

In our experiments, we consider the symmetric family of graphs with $\ell=a+2$, which we denote as $G_a$. Thus, $G_a$ consists of two copies of $K_a$ joined by
a path with $a+2$ vertices, of which exactly $a$ are
outside the cliques. Consequently,
\[
 n=3a,\qquad
 m=2\binom{a}{2}+a+1=a^2+1.
\]

This choice has a particularly interesting extremal interpretation. Breen et al.~\cite{BreenBarbell19} showed that barbell graphs can have
Kemeny's constant of cubic order in the number of vertices, which is
the largest possible asymptotic order for an $n$-vertex graph.
Subsequently, Faught, Kempton, and Knudson
~\cite{FaKeKn2022} obtained a simpler exact expression for
Kemeny's constant of barbell graphs and showed that, when $n$ is
divisible by $3$, $G_a$ with $a=n/3$
maximizes Kemeny's constant among all barbell graphs on $n$ vertices.
The family considered in our experiments is precisely this
extremal barbell family.

Moreover, by~\cite[Corollary~3.6]{FaKeKn2022},
\[
K(G)
=
\frac{1}{54}
\left[
n^3+3n^2+60n-270
+
\frac{297n^2-729n+5832}{n^3+9n}
\right].
\]
In particular,
\[
K(G)
=
\frac{n^3}{54}
+
\mathcal{O}(n^2)
=
\frac{a^3}{2}
+
\mathcal{O}(a^2).
\]

We consider the case $a=400$. Thus,
\(
n=1200,
\) \(
|E|=160001,
\)
and the exact formula above gives
\[
K(G)=32081328.3379\ldots.
\]
The results obtained by applying Algorithm~\ref{sparsifiedkemeny}
are shown in Table~\ref{tab:comp-barbell}.

\begin{table}[htbp]
\centering
\begingroup
\scriptsize
\setlength{\tabcolsep}{3.5pt}
\renewcommand{\arraystretch}{1.25}
\begin{tabular}{@{}cccccccc@{}}
\hline
$\epsilon_{\mathrm{obs}}$ &
$\beta_{\mathrm{obs}}$ &
$K(G')$ &
$K_{\beta}(G')$ &
\shortstack{Relative error\\for $K_{\beta}(G')$} &
\shortstack{Edges\\of $G'$} &
\shortstack{Edges\\removed} &
\shortstack{\% $|E|$\\removed} \\
\hline
0.200 & $2.419\cdot 10^{-9}$ &
$3.209\cdot 10^{7}$ & $3.209\cdot 10^{7}$ &
0.0003319 & 90654 & 69347 & 43.34 \\

0.487 & $2.100\cdot 10^{-9}$ &
$3.211\cdot 10^{7}$ & $3.211\cdot 10^{7}$ &
0.0009163 & 22832 & 137169 & 85.73 \\

0.783 & $1.198\cdot 10^{-9}$ &
$3.210\cdot 10^{7}$ & $3.210\cdot 10^{7}$ &
0.0008924 & 10426 & 149575 & 93.48 \\

0.797 & $1.761\cdot 10^{-9}$ &
$3.168\cdot 10^{7}$ & $3.168\cdot 10^{7}$ &
0.01247 & 10190 & 149811 & 93.63 \\

0.976 & $8.643\cdot 10^{-10}$ &
$3.262\cdot 10^{7}$ & $3.262\cdot 10^{7}$ &
0.01707 & 7470 & 152531 & 95.33 \\
\hline
\end{tabular}
\endgroup
\caption{Computational results for different sparsification levels for
the extremal barbell graph with $a=400$ and $n=1200$.}
\label{tab:comp-barbell}
\end{table}

The approximations are particularly accurate for this network, even
at very high sparsification levels. For example, when
$\epsilon_{\mathrm{obs}}\approx0.976$ and
$\beta_{\mathrm{obs}}\approx8.643\cdot10^{-10}$, more than
$95\%$ of the edges are removed, while the observed relative error in
$K_{\beta}(G')$ is only about $1.71\%$.

This is substantially smaller than the worst-case error allowed by
Theorem~\ref{theor:aprox}. Indeed, for
$\epsilon_{\mathrm{obs}}\approx0.976$ and the above value of
$\beta_{\mathrm{obs}}$, the theorem gives approximately
\[
0.012\,K(G)
\leq
K_\beta(G')
\leq
82.333\,K(G),
\]
whereas numerically
\[
K_\beta(G')
\approx
1.017\,K(G).
\]
Thus, for this extremal barbell instance, the theoretical bounds are
highly conservative, while the spectral sparsifier preserves Kemeny's
constant remarkably well.

\subsubsection{Barab\'asi--Albert graphs}

The Barab\'asi--Albert model combines network growth with preferential
attachment \cite{BA99}. In our implementation, based on \texttt{Graphs.jl} \cite{GraphsJL}, we
start from a complete graph on $n_0\geq 2$ vertices. Each new vertex is
then joined to $q$ distinct existing vertices by degree-proportional
preferential attachment, where
\[
1\leq q\leq n_0\leq n.
\]
Thus, the resulting graph is connected and the number of edges is
\[
m=\binom{n_0}{2}+q(n-n_0).
\]

The classical asymptotic Barab\'asi--Albert model is usually studied
with a fixed attachment parameter as the number of vertices tends to
infinity. In that regime, the graph is sparse and exhibits the familiar
power-law degree distribution with exponent $3$. The instance considered
here is intentionally different: we take $n_0$ and $q$ to be relatively
large in order to obtain a dense preferential-attachment network.
Therefore, this experiment is intended to test the performance of the
method on a dense network with preferential-attachment degree
heterogeneity, rather than to reproduce the classical sparse
asymptotic regime.

Kemeny's constant has previously been investigated numerically for
\\ Barab\'asi--Albert graphs. In particular, Wang et
al.~\cite{Wa17} tested their Laplacian bounds on preferential-attachment
networks and found the upper bound to be very accurate in their
experiments. More generally, Xu et al.~\cite{XSZKZ20} studied the
scaling of Kemeny's constant in several sparse scale-free and
small-world network models. These results provide useful context, but
they do not directly give an asymptotic formula for the dense
finite-size variant considered here.

We generate one realization with
\[
n=3000,\qquad n_0=300,\qquad q=300.
\]
Consequently,
\[
|E|
=
\binom{300}{2}
+
300(3000-300)
=
854850,
\]
which corresponds to an edge density of approximately $0.1900$. For
this realization,
\[
K(G)=3002.302.
\]

This value is also very close to the general lower bound
\[
K(G)\geq \frac{(n-1)^2}{n}.
\]
Indeed, since the nonzero normalized Laplacian eigenvalues satisfy
$\sum_{i=2}^n\mu_i=n$, the Cauchy--Schwarz inequality gives
\[
K(G)=\sum_{i=2}^n\frac{1}{\mu_i}
\geq
\frac{(n-1)^2}{n},
\]
with equality for the complete graph. For $n=3000$, this yields
\[
\frac{2999^2}{3000}
=
2998.0003\ldots,
\]
so the Kemeny constant of the present Barab\'asi--Albert instance is
only about $0.14\%$ above this lower bound.

The results obtained by applying Algorithm~\ref{sparsifiedkemeny} are
shown in Table~\ref{tab:comp-ba}.

\begin{table}[htbp]
\centering
\begingroup
\scriptsize
\setlength{\tabcolsep}{3.5pt}
\renewcommand{\arraystretch}{1.25}
\begin{tabular}{@{}cccccccc@{}}
\hline
$\epsilon_{\mathrm{obs}}$ &
$\beta_{\mathrm{obs}}$ &
$K(G')$ &
$K_{\beta}(G')$ &
\shortstack{Relative error\\for $K_{\beta}(G')$} &
\shortstack{Edges\\of $G'$} &
\shortstack{Edges\\removed} &
\shortstack{\% $|E|$\\removed} \\
\hline
0.213 & $1.003\cdot 10^{-8}$ &
3008.308 & 3008.308 &
0.002000 & 382045 & 472805 & 55.31 \\

0.538 & $1.018\cdot 10^{-8}$ &
3034.328 & 3034.328 &
0.01067 & 118090 & 736760 & 86.19 \\

0.892 & $1.085\cdot 10^{-8}$ &
3076.354 & 3076.354 &
0.02466 & 57178 & 797672 & 93.31 \\

0.932 & $1.086\cdot 10^{-8}$ &
3105.986 & 3105.986 &
0.03453 & 42502 & 812348 & 95.03 \\

0.989 & $1.225\cdot 10^{-8}$ &
3156.750 & 3156.750 &
0.05144 & 29940 & 824910 & 96.50 \\
\hline
\end{tabular}
\endgroup
\caption{Computational results for different sparsification levels for
the dense Barab\'asi--Albert preferential-attachment instance with
$n=3000$, $n_0=300$, and $q=300$.}
\label{tab:comp-ba}
\end{table}

The results show that substantial sparsification is again possible
while preserving Kemeny's constant with relatively small error. For
example, for
$\epsilon_{\mathrm{obs}}=0.538$, approximately $86.2\%$ of the edges
are removed, whereas the relative error in $K_{\beta}(G')$ is about
$1.07\%$. With
$\epsilon_{\mathrm{obs}}=0.892$, more than $93\%$ of the edges are
removed and the relative error remains below $2.5\%$.

Even for the most aggressive sparsification reported in
Table~\ref{tab:comp-ba}, which removes $96.5\%$ of the original
edges, the observed relative error is approximately $5.14\%$.
Thus, for this dense preferential-attachment instance, the number of
edges can be reduced by more than one order of magnitude while
retaining a reasonably accurate estimate of Kemeny's constant.

\subsubsection{The Jazz collaboration network}\label{sec:comp-jazz}

As a real-world example, we consider the Jazz collaboration network
introduced by Gleiser and Danon~\cite{GD03} and distributed
by KONECT under the identifier \texttt{arenas-jazz}
\cite{Kunegis13}. In this network, vertices represent jazz musicians
and edges represent collaborations between musicians who played in the
same band.

We use the supplied KONECT edge list as a simple, undirected,
unweighted graph, discarding repeated pairs and self-loops. The
resulting network has
$n=198$
and
$
|E|=2742,
$
and therefore edge density
\(
\dfrac{2|E|}{n(n-1)}
=
0.140594\ldots.
\)
Unlike the preceding examples, this is a fixed empirical network
rather than a realization of a generative random graph model. It
therefore provides a test of the approximation method on an observed
network with non-synthetic structural heterogeneity.

For this network,
\[
K(G)=216.470.
\]
The results obtained by applying Algorithm~\ref{sparsifiedkemeny} are
shown in Table~\ref{tab:comp-jazz}.

\begin{table}[htbp]
\centering
\begingroup
\scriptsize
\setlength{\tabcolsep}{3.5pt}
\renewcommand{\arraystretch}{1.25}
\begin{tabular}{@{}cccccccc@{}}
\hline
$\epsilon_{\mathrm{obs}}$ &
$\beta_{\mathrm{obs}}$ &
$K(G')$ &
$K_{\beta}(G')$ &
\shortstack{Relative error\\for $K_{\beta}(G')$} &
\shortstack{Edges\\of $G'$} &
\shortstack{Edges\\removed} &
\shortstack{\% $|E|$\\removed} \\
\hline
0.235 & $1.091\cdot 10^{-8}$ &
217.102 & 217.102 &
0.002919 & 2526 & 216 & 7.88 \\

0.483 & $1.085\cdot 10^{-8}$ &
220.103 & 220.103 &
0.01679 & 1776 & 966 & 35.23 \\

0.532 & $1.044\cdot 10^{-8}$ &
219.193 & 219.193 &
0.01258 & 1983 & 759 & 27.68 \\

0.672 & $1.127\cdot 10^{-8}$ &
223.964 & 223.964 &
0.03462 & 1388 & 1354 & 49.38 \\

0.920 & $9.897\cdot 10^{-8}$ &
229.180 & 229.180 &
0.05872 & 1011 & 1731 & 63.13 \\
\hline
\end{tabular}
\endgroup
\caption{Computational results for different sparsification levels for
the Jazz collaboration network.}
\label{tab:comp-jazz}
\end{table}

The amount of sparsification achieved for this empirical network is
more moderate than for the dense synthetic examples considered above.
For instance, when $\epsilon_{\mathrm{obs}}=0.483$, approximately
$35.2\%$ of the edges are removed and the relative error in
$K_{\beta}(G')$ is about $1.68\%$. For
$\epsilon_{\mathrm{obs}}=0.672$, almost half of the edges are removed,
while the relative error remains below $3.5\%$.

At the strongest sparsification level reported in
Table~\ref{tab:comp-jazz}, $63.13\%$ of the edges are removed and the
relative error is approximately $5.87\%$. Thus, spectral
sparsification still provides a substantial reduction in network size
for this real-world example, although the achievable reduction is less
pronounced than for the much denser synthetic networks considered
previously.

\subsubsection{Chung--Lu graphs with power-law weights}

Finally, we consider a Chung--Lu-type random graph with heterogeneous
power-law weights. The Chung--Lu model generates independent edges from
a prescribed sequence of vertex weights~\cite{CL02}. We use the
capped simple-graph variant
\[
\theta_i
=
c_0 i^{-1/(\alpha-1)},
\qquad
i=1,\ldots,n,
\]
with edge probabilities
\[
p_{ij}
=
\min\left\{
1,\,
\frac{\theta_i\theta_j}
{\sum_{h=1}^n\theta_h}
\right\},
\qquad i<j,
\]
where $\alpha>1$ and $c_0>0$. Every selected edge has unit
conductance; the quantities $\theta_i$ are parameters of the random
graph model and should not be confused with edge conductances.

The constant $c_0$ is determined numerically by bisection so that the
actual expected average degree of the capped model satisfies
\[
\frac{2}{n}
\sum_{1\leq i<j\leq n}p_{ij}
=
\overline d
\]
to relative tolerance $10^{-8}$. Thus, the calibration explicitly
takes the probability cap into account.

We use the parameters
\[
n=2000,
\qquad
\alpha=2.05,
\qquad
\overline d=0.1(n-1),
\]
so that the expected edge density is $0.1$, up to the calibration
tolerance.

An important distinction from the standard uncapped Chung--Lu model is
that, in the present construction, in which the graph is loopless, the parameters $\theta_i$
do not coincide exactly with the expected degrees even when the
cap is inactive, since
\[
\mathbb E[k_i]
=\theta_i\left(1-\frac{\theta_i}{\sum_h\theta_h}\right).
\]
When the cap $p_{ij}=1$ is active, the discrepancy can be much larger. For this reason,
we refer to the present model as a Chung--Lu graph with power-law
weights, without asserting that the realized degree sequence follows
an exact power law. The role of admissibility conditions and
finite-size effects in power-law Chung--Lu constructions is discussed
in~\cite{FTT21}.

For the realization used in our experiment,
\(
|E|=199976,
\)
corresponding to an average degree
\(
\frac{2|E|}{n}=199.976
\)
and an observed edge density
\(
\frac{2|E|}{n(n-1)}
=
0.100038\ldots.
\)
Its Kemeny's constant is
\[
K(G)=2005.028.
\]

As in the dense Barab\'asi--Albert example, this value is close to the
general lower bound
\(
K(G)\geq\dfrac{(n-1)^2}{n}.
\)
For $n=2000$,
\(
\dfrac{1999^2}{2000}
=
1998.0005,
\)
so the Kemeny constant of this realization is approximately
$0.35\%$ above the minimum possible value for a connected graph of
the same order.

The results obtained by applying Algorithm~\ref{sparsifiedkemeny} are
shown in Table~\ref{tab:comp-chunglu}.

\begin{table}[htbp]
\centering
\begingroup
\scriptsize
\setlength{\tabcolsep}{3.5pt}
\renewcommand{\arraystretch}{1.25}
\begin{tabular}{@{}cccccccc@{}}
\hline
$\epsilon_{\mathrm{obs}}$ &
$\beta_{\mathrm{obs}}$ &
$K(G')$ &
$K_{\beta}(G')$ &
\shortstack{Relative error\\for $K_{\beta}(G')$} &
\shortstack{Edges\\of $G'$} &
\shortstack{Edges\\removed} &
\shortstack{\% $|E|$\\removed} \\
\hline
0.365 & $1.029\cdot10^{-8}$ &
2012.987 & 2012.987 &
0.003969 & 83371 & 116605 & 58.31 \\

0.495 & $8.234\cdot10^{-9}$ &
2018.949 & 2018.949 &
0.006943 & 60844 & 139132 & 69.57 \\

0.840 & $8.476\cdot10^{-9}$ &
2028.934 & 2028.934 &
0.01192 & 42602 & 157374 & 78.70 \\

0.977 & $8.442\cdot10^{-9}$ &
2065.149 & 2065.149 &
0.02998 & 21126 & 178850 & 89.44 \\

0.982 & $1.377\cdot10^{-8}$ &
2049.398 & 2049.398 &
0.02213 & 27108 & 172868 & 86.44 \\
\hline
\end{tabular}
\endgroup
\caption{Computational results for different sparsification levels for
the Chung--Lu network with power-law weights.}
\label{tab:comp-chunglu}
\end{table}

The results show that Kemeny's constant remains accurately approximated
even after substantial sparsification. For example, when
$\epsilon_{\mathrm{obs}}=0.495$, almost $70\%$ of the edges are
removed while the relative error in $K_{\beta}(G')$ is below $0.7\%$.
For $\epsilon_{\mathrm{obs}}=0.840$, approximately $78.7\%$ of the
edges are removed and the relative error remains close to $1.2\%$.

The two most aggressive sparsifications remove more than $86\%$ of
the original edges, while the corresponding relative errors remain
below $3\%$. In particular, for
$\epsilon_{\mathrm{obs}}=0.977$, approximately $89.4\%$ of the edges
are removed and the relative error is about $3.0\%$. Notice that the percentage of removed edges is not necessarily
monotone in $\epsilon_{\mathrm{obs}}$.

Taken together, the numerical experiments show that the proposed
method behaves consistently across networks with very different
structural properties: dense homogeneous random graphs, graphs with
a pronounced bottleneck, \\preferential-attachment networks, a
real-world collaboration network, and random graphs with strongly
heterogeneous weights. In all cases, the error introduced by the
approximate Laplacian solves is negligible compared with the
sparsification error, while substantial reductions in the number of
edges can be achieved without a comparable deterioration in the
approximation of Kemeny's constant.

\subsection*{AI statement}

GPT-5.6 Sol (OpenAI; accessed October 2026) has been used to polish the writing and for assistance in proofreading. AI has not been used for any mathematical reasoning.

\subsection*{Acknowledgements}

Aida Abiad is supported by the Dutch Research Council (NWO) through
grant VI.Vidi.213.085. This work has been partially supported by the
Agencia Estatal de Investigaci\'on (AEI) under project
PID2021-122501NB-I00 and by the Universitat Polit\`ecnica de Catalunya
under funds AGRUPS-UPC 2026. \'A. Samperio was partially supported by
grant PID2022-138906NB-C21, funded by
MICIU/AEI/10.13039/501100011033 and by ERDF ``A way of making Europe''.

\end{document}